\documentclass[12pt]{article}

\usepackage{setspace}

\usepackage{amsthm, amssymb, amstext}
\usepackage[fleqn]{amsmath}
\usepackage{latexsym}
\usepackage[dvips]{graphicx}
\usepackage{comment}
\usepackage{hyperref}
\usepackage{mathtools}
\usepackage{enumerate} 

\usepackage{todonotes}
\usepackage{comment}

\newtheorem{theorem}{Theorem}
\newtheorem{lemma}[theorem]{Lemma}
\newtheorem{observation}[theorem]{Observation}

\newtheorem{conjecture}[theorem]{Conjecture}

\usepackage{amsthm}

\title{A Thomassen-type method for \\planar graph recoloring}
\author{Zden\v{e}k Dvo\v{r}\'ak\thanks{Computer Science Institute of Charles University, Prague, Czech Republic, email: \texttt{rakdver@iuuk.mff.cuni.cz} }\and
Carl Feghali\thanks{Computer Science Institute of Charles University, Prague, Czech Republic, email: \texttt{feghali.carl@gmail.com} }}

\date{}

\begin{document}
\maketitle

\begin{abstract}
The reconfiguration graph $R_k(G)$ for the $k$-colorings of a graph $G$ has as vertices all
possible $k$-colorings of $G$ and two colorings are adjacent if they differ in the color of exactly one vertex. 
We use a list coloring technique inspired by results of Thomassen to prove that for a planar graph $G$ with $n$
vertices, $R_{10}(G)$ has diameter at most $8n$, and if $G$ is triangle-free, then $R_7(G)$ has diameter at most~$7n$.
 \end{abstract}
 
\section{Introduction}

How far are the proper colorings of a graph from one another?  More precisely, how many simple operations of some
kind are necessary to change one coloring to any other coloring, with all intermediate colorings also being proper?
Questions of this kind arise for example in
the statistical physics (where the colorings represent physical states of the studied system and the operations correspond
to valid transitions between the states), when sampling a random coloring of a graph, or when estimating the number of colorings of a graph.
We refer the reader to the surveys by van den Heuvel~\cite{He13} and by Nishimura~\cite{nishimura} for more
background.

In this paper, we consider the reconfiguration model 
(also known as \emph{Glauber dynamics} in statistical physics), where the allowed operation is changing the color of
a single vertex at a time.  Let us give a precise definition.
Let $G$ be a graph, and let $k$ be a non-negative integer. 
A (proper) \emph{$k$-coloring} of $G$ is a function $\varphi: V(G) \rightarrow \{1, \dots, k\}$ such that $\varphi(u) \neq \varphi(v)$ whenever $uv \in E(G)$. 
The reconfiguration graph $R_k(G)$ for the $k$-colorings of $G$ has as vertices all $k$-colorings of $G$,
with two colorings adjacent if they differ in the color of exactly one vertex.  That is, two $k$-colorings $\varphi_1$ and $\varphi_2$ are joined by a path
in $R_k(G)$ if and only if we can transform $\varphi_1$ into $\varphi_2$ by
recoloring vertices one by one, always keeping the coloring proper. 

In this setting, one can ask a number of questions.  Under which conditions is $R_k(G)$ connected?
If it is, is its diameter polynomial in $|V(G)|$?  Again, we refer the reader to the aforementioned surveys
for more details.  In this paper, we ask for conditions ensuring that the diameter of $R_k(G)$ is at most
linear in $|V(G)|$.  For example, in terms of degeneracy, we have the following result
(recall that a graph is \emph{$d$-degenerate} if every subgraph of the graph contains a vertex of degree at most $d$).
\begin{theorem}[{Bousquet and Perarnau~\cite[Theorem~1]{bousquet11}}]\label{thm:bousquet}
If $G$ is a $d$-degenerate graph on $n$ vertices and $c\ge 2d+2$, then $R_c(G)$ has diameter at most $(d+1)n$. 
\end{theorem}
Let $P_{d,n}$ be the graph consisting of a path with $n-d+1$ vertices and of a clique of $d-1$ vertices adjacent
to all the vertices of the path.  Note that $P_{d,n}$ is $d$-degenerate.
On the negative side, Bonamy et al.~\cite{BJLPP14} proved that for every fixed $d$, $R_{d+2}(P_{d,n})$ has diameter
quadratic in $n$.  Thus, if we define $c_d$ as the minimum number of colors such that for every $d$-degenerate graph $G$,
the reconfiguration graph $R_{c_d}(G)$ has diameter $O(|V(G)|)$, we have $d+3\le c_d\le 2d+2$.

The upper bound can be improved in the more restricted setting of planar graphs.
Planar graphs are 5-degenerate, and thus Theorem~\ref{thm:bousquet} implies 12 colors suffice to ensure linear diameter
of their reconfiguration graphs. 
The first improvement comes from the following general bound exploiting list colorability
(the bound is implicit in~\cite{feghali10}; for completeness, we give a proof in the next section).
A \emph{list assignment} $L$ for a graph $G$ is a function that to each vertex $v\in V(G)$
assigns a set $L(v)$ of colors.  An \emph{$L$-coloring} of $G$ is a proper
coloring $\varphi$ of $G$ such that $\varphi(v) \in L(v)$ for each $v \in V(G)$.
\begin{lemma}\label{lemma:degchos}
Let $c$ and $d$ be positive integers.  Suppose $G$ is a $d$-degenerate graph on $n$ vertices.
If $G$ is $L$-colorable from every assignment $L$ of lists of size $c$, then for every independent
set $I$ in $G$, the diameter of $R_{c+d+1}(G)$ is at most $2(n+|I|)$ plus the diameter of $R_{c+d}(G-I)$.
\end{lemma}
Note that planar graphs are colorable from any assignment of lists of size five~\cite{thomassen1}.
Moreover, every $d$-degenerate graph $G$ has an independent set $I$ such that $G-I$ is $(d-1)$-degenerate.
Thus, Lemma~\ref{lemma:degchos} together with Theorem~\ref{thm:bousquet} imply that $R_{11}(G)$
has linear diameter for every planar graph $G$.  We have recently decreased this bound to 10.
\begin{theorem}[Dvo\v{r}\'ak and Feghali~\cite{DF}]\label{thm:main}
If $G$ is a planar graph on $n$ vertices, then $R_{10}(G)$ has diameter at most $8n$. 
\end{theorem}
The proof in~\cite{DF} is based on discharging and reducible configurations, and is quite long and involved.
In this paper, we present a shorter proof based on a different idea.  Furthermore, we apply the same
idea to triangle-free planar graphs.  These graphs are 3-degenerate, and thus by Theorem~\ref{thm:bousquet}, $R_8(G)$
has linear diameter for every triangle-free planar graph $G$.  We decrease the bound on the number of colors to 7.
\begin{theorem}\label{thm:trmain}
If $G$ is a planar triangle-free graph on $n$ vertices, then $R_7(G)$ has diameter at most $7n$. 
\end{theorem}
It is natural to ask what happens for graphs of girth at least five.  Since planar graphs of girth at least five can be colored
from lists of size three~\cite{thomassen3}, Lemma~\ref{lemma:degchos} together with Theorem~\ref{thm:bousquet} implies $7$ colors suffice,
by an argument substantially simpler than the one used to establish Theorem~\ref{thm:trmain}.  We believe the bound can be improved
to $6$.
\begin{conjecture}\label{conj:girth5}
If $G$ is a planar graph of girth at least five on $n$ vertices, then $R_6(G)$ has diameter $O(n)$.
\end{conjecture}
We suspect a variation on our idea can be used to prove Conjecture~\ref{conj:girth5}, but a number of complications arise in this case.
Note that for planar graphs of girth at least six, this bound follows by $2$-degeneracy and Theorem~\ref{thm:bousquet}.

\section{The proof idea}

Due to the following standard lemma, whose proof we include for completeness, it suffices to free up one of the colors.
\begin{lemma}\label{lemma:freeup}
Let $k$ and $m$ be positive integers and let $G$ be a graph.  Suppose that for every $k$-coloring $\alpha$
of $G$, there exists a $k$-coloring $\alpha'$ such that $\alpha'$ only uses colors $\{1,\ldots,k-1\}$
and the distance from $\alpha$ to $\alpha'$ in $R_k(G)$ is at most $m$.  Then for every independent set $I$ in $G$,
the diameter of $R_k(G)$ is at most $2(m+|I|)$ plus the diameter of $R_{k-1}(G-I)$.
\end{lemma}
\begin{proof}
Consider any $\alpha,\beta\in V(R_k(G))$.  By the assumptions, there exist colorings $\alpha',\beta'\in V(R_k(G))$ at distance at most
$m$ from $\alpha$ and $\beta$, respectively, not using the color $k$. To bound the distance between $\alpha'$ and $\beta'$,
first recolor all vertices of $I$ to the color $k$ in both colorings, then recolor the vertices of $V(G)\setminus I$
according to the path from $\alpha'\restriction V(G-I)$ to $\beta'\restriction V(G-I)$ in $R_{k-1}(G-I)$.
The recolorings according to this path do not use the color $k$, and thus do not conflict with the coloring of $I$.
\end{proof}

As an example application, let us prove Lemma~\ref{lemma:degchos}.

\begin{proof}[Proof of Lemma~\ref{lemma:degchos}]
Consider any $(c+d+1)$-coloring $\alpha$ of $G$.
Since $G$ is $d$-degenerate, there exists an ordering $v_1$, \ldots, $v_n$ of its vertices
such that for each $i$, $v_i$ has at most $d$ neighbors $v_j$ such that $j>i$.  For $i=1,\ldots, n$,
let
$$L(v_i)=\{1,\ldots,c+d\}\setminus \{\alpha(v_j):v_iv_j\in E(G),j>i\}.$$
Note that $|L(v_i)|\ge c$ for each $i$, and by the assumptions, $G$ has an $L$-coloring $\alpha'$.
For $i=1, \ldots, n$ in order, recolor $v_i$ to $\alpha'(v_i)$.  The color assigned to $v_i$ does not
conflict with the neighbors $v_j$ with $j<i$ since $\alpha'$ is a proper coloring, and with those with $j>i$
by the choice of the list assignment $L$.  Hence, the coloring $\alpha'$, which does not use the color $c+d+1$,
is at distance at most $n$ from $\alpha$ in $R_{c+d+1}$.  Lemma~\ref{lemma:degchos} thus follows from Lemma~\ref{lemma:freeup}.
\end{proof}

For our improved bounds, we show how to eliminate one of the colors in the following more general list coloring setting.
\begin{theorem}\label{thm:col}
Let $G$ be a planar graph, let $L$ be a list assignment for $G$,
and let $\varphi$ be an $L$-coloring of $G$.  Let $f$ be an arbitrary function assigning a color to each vertex of $G$.
If either
\begin{itemize}
\item[(a)] $|L(v)|\ge 10$ for every $v\in V(G)$, or
\item[(b)] $G$ is triangle-free and $|L(v)|\ge 7$ for every $v\in V(G)$,
\end{itemize}
then there exists an $L$-coloring $\varphi'$ of $G$ such that $\varphi'(v)\neq f(v)$ for each $v\in V(G)$ and
$\varphi$ can be transformed to $\varphi'$ by recoloring vertices one by one so that
all intermediate colorings are proper $L$-colorings and each vertex is recolored at most twice.
\end{theorem}

Theorems~\ref{thm:main} and~\ref{thm:trmain} then easily follow.

\begin{proof}[Proof of Theorem~\ref{thm:main} and~\ref{thm:trmain}]
Let $k=7$ and $p=3$ if $G$ is triangle-free, and $k=10$ and $p=4$ otherwise.
Theorem~\ref{thm:col} applied with the list assignment $L$ giving each vertex the list $\{1,\ldots,k\}$
and the function $f$ assigning the color $k$ to each vertex shows that
for every $k$-coloring $\alpha$ of $G$, there exists a $k$-coloring $\alpha'$ such that $\alpha'$ only uses colors $\{1,\ldots,k-1\}$
and the distance from $\alpha$ to $\alpha'$ in $R_k(G)$ is at most $2n$.

Since $G$ is planar, the result of Thomassen~\cite{thomassen} shows that $G$ contains an independent set $I$ such that $G-I$ is $3$-degenerate.
Moreover, if $G$ is triangle-free, then $G$ is $3$-degenerate, and thus it contains an independent set $I$ such that $G-I$ is $2$-degenerate.
In either case, $R_{k-1}(G-I)$ has diameter at most $p(n-|I|)$ by Theorem~\ref{thm:bousquet}.
Therefore, by Lemma~\ref{lemma:freeup}, $R_k(G)$ has diameter at most
$2(2n+|I|)+p(n-|I|)=(4+p)n+(2-p)|I|\le (4+p)n$, as required.
\end{proof}

How to prove Theorem~\ref{thm:col}?  The first idea is to adapt the proof of Thomassen~\cite{thomassen1} for colorability of planar
graphs from lists of size five: Recolor $G$ by small pieces nibbled from the outer face boundary, and exploit weaker assumptions
on the list sizes on the outer face boundary to prevent conflicts with previously recolored parts. 

That is, we could allow each vertex on the outer face boundary to have only a list of size $8$ rather than $10$.  Then, we could try to
eliminate a vertex $v$ incident with the outer face.  Let us say we decide on a new color $c\neq f(v)$ for $v$.  It may not be immediately possible
to recolor $v$ to $c$, as some of its neighbors can use the color $c$.  However, we can remove $\varphi(v)$ and $c$ from the lists
of the neighbors of $v$ and apply induction to $G-v$.  Whatever recolorings we perform on $G-v$ do not conflict with $v$, as $\varphi(v)$
was removed from the lists of the neighbors of $v$.  Moreover, in the final coloring, no neighbor of $v$ has color $c$, and thus
we can finish by recoloring $v$ to $c$.  Deleting the two colors $\varphi(v)$ and $c$ from the lists of size $10$ reduces the list
sizes to $8$, which is fine as the neighbors of $v$ are incident with the outer face of $G-v$.

Unfortunately, there is a gap in the above argument: We also have to delete the colors from the list of a neighbor $u$ of $v$
adjacent to it in the boundary of the outer face, decreasing its size below $8$.  We might think about solving this issue by choosing
the color $c$ not belonging to $L(u)$ (similarly to~\cite{thomassen3}) or by reserving several colors for $v$ to choose from rather
than just a single color $c$ and not removing them from $L(u)$ (similarly to~\cite{thomassen1}), but in any case we do not have any control over the removal of the color $\varphi(v)$.

We solve this issue by choosing the final color for all vertices of the cycle~$K$ bounding the outer face at once,
applying induction to $G-V(K)$ with reduced lists, and finally recoloring the vertices of $K$ to their chosen colors in some order.
This brings another problem, though: A vertex $x\in V(G)\setminus V(K)$ can have many neighbors in $K$, and thus the size of the list
of $x$ could be reduced too much (again, regardless of the choice of the final colors on $K$, we have no control over the colors assigned
to $K$ by $\varphi$).

A standard way of dealing with this issue is by a precoloring extension argument.  In our setting, we could aim for the
following claim.  Suppose $|L(v)|\ge 6$ for each vertex $v\in V(K)$ and $|L(v)|\ge 10$ for $v\in V(G)\setminus V(K)$.
Moreover, suppose a sequence of recolorings $\sigma$ is prescribed for a $3$-vertex path $P\subset K$.  Then there exists a sequence
of recolorings of~$G$ from the given lists, where the final color of each vertex $v$ is different from $f(v)$,
and where the vertices of~$P$ are recolored according to $\sigma$ \emph{after} all other recolorings of~$G$.
Now, if every vertex $x\in V(G)\setminus V(K)$ has at most two neighbors in $K$, we can proceed as in the previous paragraph:
For each neighbor $v$, delete from $L(x)$ the color $\varphi(v)$ and the desired final color of $v$ (as $x$ has at most two neighbors in $K$
and $|L(x)|=10$, the resulting list of $x$ has size at least $6$), and apply the induction to $G-V(K)$.  Otherwise, $G$ can be split into two graphs $G_1$ and $G_2$ intersecting in a path $v_1xv_2$,
where $v_1$ and $v_2$ are non-adjacent neighbors of $x$ in $K$.  For simplicity, let us ignore the possible complications if say
$v_1$ is the middle vertex of $P$, and suppose by symmetry that $P\subseteq G_1$.  Then we can apply induction to
$G_1$, obtaining a sequence $\sigma_1$ of recolorings.  Next, we apply induction to $G_2$, with $\sigma_1$ prescribing a sequence of
recolorings of the path $v_1xv_2$, obtaining a sequence $\sigma_2$ of recolorings.  Finally, we can alter these to a sequence of
recolorings of $G$: First, perform the recolorings in the sequence $\sigma_2$ restricted to $G_2-\{v_1,x,v_2\}$, which is possible
since the recolorings of the path $v_1xv_2$ are performed last in $\sigma_2$.  Then, perform $\sigma_1$, which is possible since
the recolorings of the path $v_1xv_2=G_1\cap G_2$ in $\sigma_1$ are the same as those at the end of $\sigma_2$.

This promising idea unfortunately fails when we start to consider the details.  The basic case is that $G$ consists just of
a path $P$ and a vertex $v$ adjacent to all three vertices of $P$.  Then $v$ has only five colors to choose from
(those in $L(v)\setminus\{f(v)\}$), but also has to avoid all the colors that appear on $P$ in its prescribed recoloring sequence;
and there can be up to six such colors (the original and the final color for each vertex of $P$).  Considering the previous
paragraph, it turns out we have some control over the final color of the middle vertex of $P$ (we can remove some colors from
the list of $x$ before applying induction to $G_1$); still, there are five forbidden colors over which we have little to no control.

To avoid this issue, we use a more complicated 2-phase recoloring process, where in each phase, we allow to recolor each vertex at most once.
In the first phase, vertices $v\in V(K)$ will have lists of size $6$, but we do not forbid them from being recolored to the color $f(v)$.
In the second phase, they will have lists of size $8$.  Considering the previous paragraph, in both phases $v$ has at least $6$ possible final
colors, while there are only five colors used on $P$ over which we do not have any control, giving us a hope the recoloring might succeed.
For the part of the argument in the case that no vertex in $V(G)\setminus V(K)$ has more than two neighbors in $K$, the key insight is
that with some care, it is possible to recolor the vertices of $K$ only once, in the first phase, and leave them untouched in the
second phase.  Thus, while in the first phase, we might need to remove up to four colors from the list of each vertex $x\in V(G)\setminus V(K)$,
in the second phase we are only removing at most two colors (for each neighbor $v\in V(K)$ of $x$, the color of $v$ after the first phase,
which stays unchanged in the second phase).

Let us remark on a slight discrepancy between the exposition in the previous paragraph and the rest of the paper:
For convenience, we directly remove the color $f(v)$ from the second phase lists, thus leading to the assumption (see (Ga) below) that the
second phase lists have size at least $7$ rather than~$8$.

\section{Preliminaries}

Let $G$ be a graph.
A proper coloring $\varphi'$ of $G$ is obtained from another coloring $\varphi$ by \emph{recoloring} if $\varphi$ and $\varphi'$
differ in the color of exactly one vertex.
A \emph{sequence of recolorings} is a sequence $\sigma=(v_1,c_1), \ldots, (v_m,c_m)$, where for $i=1,\ldots,m$
$v_i$ is a vertex of $G$ whose color is changed to $c_i$.  By $\varphi+\sigma$, we denote the coloring after
performing these changes; unless specified otherwise, we require that all the intermediate colorings are proper; when we want to make this
more explicit, we say that this is a \emph{proper} sequence of recolorings.
We also use $+$ to indicate concatenation of sequences.
A sequence of recolorings of $G$ is said to be a \emph{once-only} if every vertex of $G$ is recolored at most once.
For an induced subgraph $H$ of $G$, we write $\sigma^{H}$ for the subsequence of $\sigma$ consisting of the recolorings on $V(H)$.
We say the sequence $\sigma$ of recolorings of $G$ is \emph{$H$-late}
if the recolorings of the vertices of $H$ appear after all recolorings of vertices of $V(G)\setminus V(H)$ in the sequence,
i.e., $\sigma=\sigma^{G-V(H)}+\sigma^H$.

Let $L_1$ and $L_2$ be list assignments for a graph $G$.
An \emph{$(L_1, L_2)$-trajectory (starting with $\varphi_0$)} in $G$ is a triple $\vec{\varphi}=(\varphi_0,\varphi_1, \varphi_2)$
where $\varphi_0$ is a proper coloring of $G$ and for $i\in\{1,2\}$, $\varphi_i$ is an $L_i$-coloring of $G$
obtained from $\varphi_{i-1}$ by a once-only sequence $\sigma_i$ of recolorings.  We say that $(\sigma_1,\sigma_2)$
is a \emph{witness} of the trajectory.

Let $H$ be an induced subgraph of $G$ and let $\vec{\delta}$ be an $(L_1, L_2)$-trajectory in $H$.
We say that $\vec{\delta}$ \emph{lifts} to $\vec{\varphi}$ if for $i\in\{0,1,2\}$, $\delta_i$ is the restriction
of $\varphi_i$ to $V(H)$, and for $i\in \{1,2\}$, there exists a once-only $H$-late sequence $\sigma_i$ of recolorings from $\varphi_{i-1}$
to $\varphi_i$.  Again, we say that $(\sigma_1,\sigma_2)$ is a witness that $\vec{\delta}$ lifts to $\vec{\varphi}$.

Importantly, it turns out that in both $\sigma_1$ and $\sigma_2$, the recolorings on $H$ can be performed in any order that is valid when
considering $H$ alone, as stated more precisely in the following observation.
\begin{observation}\label{obs:reorder}
Let $H$ be an induced subgraph of a graph $G$ and let $L_1$ and $L_2$ be list assignments for $G$.  Let $\vec{\delta}$ be an $(L_1, L_2)$-trajectory
in $H$ witnessed by $(\sigma_1,\sigma_2)$.  If $\vec{\delta}$ lifts to an $(L_1, L_2)$-trajectory $\vec{\varphi}$
in $G$, then this is witnessed by a pair $(\pi_1,\pi_2)$ such that $\pi_1^H=\sigma_1$ and $\pi_2^H=\sigma_2$.
\end{observation}
\begin{proof}
Let $(\omega_1,\omega_2)$ be a witness that $\vec{\delta}$ lifts to
the $(L_1, L_2)$-trajectory $\vec{\varphi}$, and for $i\in\{1,2\}$, let $\pi_i=\omega_i^{G-V(H)}+\sigma_i$.
Let us argue that $(\pi_1,\pi_2)$ is also a witness of this lifting.  Note that
for $i\in\{1,2\}$, the effect of the recolorings $\sigma_i$ and $\omega_i^H$ is the same;
for each $v\in V(H)$ such that $\delta_{i-1}(v)\neq \delta_i(v)$, the color of $v$ is changed
from $\delta_{i-1}(v)$ to $\delta_i(v)$. Therefore, we only need to argue that the sequence $\pi_i$
is proper on $\varphi_{i-1}$.  Since $\pi_i$ and $\omega_i$ share the initial segment, it suffices to show that
the sequence $\sigma_i$ is proper on $\varphi'_{i-1}=\varphi_{i-1}+\omega_i^{G-V(H)}$.

Consider any edge $uv\in E(G)$ with $v\in V(H)$ at the moment the sequence $\sigma_i$ recolors
$v$ to the color $\delta_i(v)$.  If $u\in V(H)$, then $u$ does not have the color $\delta_i(v)$,
since $\sigma_i$ is a proper sequence of recolorings from $\delta_{i-1}$ on $H$.  If $u\not\in V(H)$,
then $\varphi'_{i-1}(u)\neq \delta_i(v)$, since the sequence $\omega_i^H$ also recolors $v$ to $\delta_i(v)$
at some point.
\end{proof}

Suppose $\vec{\varphi}$ is an $(L_1, L_2)$-trajectory in a graph $G$.
For $F\subseteq G$, let $\vec{\varphi}\restriction F$ denote the triple
$(\varphi_0\restriction V(F), \varphi_1\restriction V(F), \varphi_2\restriction V(F))$, and note that
$\vec{\varphi}\restriction F$ is an $(L_1,L_2)$-trajectory in $F$.
We frequently use the following result that enables us to combine two trajectories.

\begin{lemma}\label{lem:total}
Let $H$ be an induced subgraph of a graph $G$.  Suppose $G=G_1\cup G_2$ for induced subgraphs $G_1$
and $G_2$ of $G$ and $H\subseteq G_1$.  Let $\varphi_0$ be a proper coloring of $G$,
let $L_1$ and $L_2$ be list assignments for $G$, and let $\vec{\delta}$ be an $(L_1,L_2)$-trajectory in $H$. If
\begin{itemize}
\item $\vec{\delta}$ lifts to an $(L_1,L_2)$-trajectory $\vec{\psi}$ in $G_1$ starting with $\varphi_0\restriction V(G_1)$, and
\item $\vec{\psi}\restriction (G_1\cap G_2)$ lifts to an $(L_1,L_2)$-trajectory $\vec{\theta}$ in $G_2$ starting with $\varphi_0\restriction V(G_2)$,
\end{itemize}
then $\vec{\delta}$ lifts to an $(L_1,L_2)$-trajectory in $G$ starting with $\varphi_0$.
\end{lemma}
\begin{proof}
Note that $\psi_0=\varphi_0\restriction V(G_1)$ and $\theta_0=\varphi_0\restriction V(G_2)$.
Since $\vec{\psi}\restriction (G_1\cap G_2)$ lifts to $\vec{\theta}$, we have
$\psi_i\restriction V(G_1\cap G_2)=\theta_i\restriction V(G_1\cap G_2)$
for $i\in\{0,1,2\}$.  Let $\vec{\varphi}=(\varphi_0,\varphi_1,\varphi_2)$, where $\varphi_i=\psi_i\cup \theta_i$ for $i\in\{1,2\}$.
We claim that $\vec{\varphi}$ is an $(L_1,L_2)$-trajectory in $G$ and $\vec{\delta}$ lifts to it.

Let $(\sigma_1,\sigma_2)$ witness that $\vec{\delta}$ lifts to $\vec{\psi}$, and let $(\pi_1,\pi_2)$ witness
that $\vec{\psi}\restriction (G_1\cap G_2)$ lifts to $\vec{\theta}$.  Note that $(\sigma_1^{G_1\cap G_2},\sigma_2^{G_1\cap G_2})$
is a witness of the $(L_1,L_2)$-trajectory $\vec{\psi}\restriction (G_1\cap G_2)$,
and thus by Observation~\ref{obs:reorder}, we can assume $\sigma_i^{G_1\cap G_2}=\pi_i^{G_1\cap G_2}$ for $i\in\{1,2\}$.
Then $(\pi_1^{G_2-V(G_1)}+\sigma_1, \pi_2^{G_2-V(G_1)}+\sigma_2)$ witnesses that $\vec{\delta}$ lifts to $\vec{\varphi}$.
Indeed, for $i\in\{1,2\}$, since $\pi_i$ is $(G_1\cap G_2)$-late, the restrictions of these sequences to $G_1$ and $G_2$
match $\sigma_i$ and $\pi_i$, respectively, ensuring all intermediate colorings are proper.
\end{proof}
Throughout the rest of the proof, we will several times need arguments analogous to those
used in the proofs of Observation~\ref{obs:reorder} and Lemma~\ref{lem:total}.  We will
generally just describe the resulting trajectories and their witnesses and leave the straightforward
verification of their validity to the reader.

A \emph{scene} $S=(G,H,L_1,L_2,\varphi_0,\vec{\delta})$ consists of
\begin{itemize}
\item a plane graph $G$ and an induced subgraph $H$ of $G$, with $V(H)$ consisting of 
vertices consecutive in the boundary of the outer face of $G$,
\item list assignments $L_1$ and $L_2$ for $G$ and a proper coloring $\varphi_0$ of $G$, and
\item an $(L_1,L_2)$-trajectory $\vec{\delta}$ in $H$ starting with $\varphi_0\restriction V(H)$.
\end{itemize}
We say the scene $S$ \emph{has a trajectory} if $\vec{\delta}$ lifts to an $(L_1,L_2)$-trajectory $\vec{\varphi}$ in $G$ starting with $\varphi_0$;
in that case, $\vec{\varphi}$ is a \emph{trajectory of $S$}.

\section{General planar graphs}

A vertex of a plane graph is \emph{internal} if it is not incident with the outer face.
A scene $S=(G,H,L_1,L_2,\varphi_0,\vec{\delta})$ is \emph{valid} if $|V(H)|\le 3$ and
\begin{itemize}
\item[(Ga)] for each $v\in V(G)$, $|L_1(v)\setminus L_2(v)|\le 1$, $|L_1(v)|\ge 6$ and $|L_2(v)|\ge 7$,
\item[(Gb)] if $v\in V(G)$ is internal, then $|L_1(v)|\ge 10$ and $|L_2(v)|\ge 9$, and
\item[(Gc)] if $|V(H)|=3$ and a vertex $v\in V(G)$ is adjacent to all vertices of $H$, then
$$L_1(v)\neq \bigcup_{u\in V(H)} \{\delta_0(u),\delta_1(u)\}.$$
\end{itemize}
We aim to show by contradiction that every valid scene has a trajectory; from this result, it is easy to derive
the case (a) of Theorem~\ref{thm:col}. 
A \emph{counterexample} is a valid scene without a trajectory.  We say $S$ is a \emph{minimal counterexample}
if $S$ is a counterexample such that $|V(G)|+|E(G)|$ is minimum among all counterexamples,
and subject to that, $|V(H)|$ is maximum.
Let us start by studying the properties of a hypothetical minimal counterexample.

\begin{lemma}\label{lemma:basprop}
Suppose $S=(G,H,L_1,L_2,\varphi_0,\vec{\delta})$ is a minimal counterexample.
Then $G$ is $2$-connected, $H$ is a $3$-vertex path, every chord of the outer
face of $G$ is incident with the middle vertex of $H$, and every triangle in $G$ bounds a face.
\end{lemma}
\begin{proof}
Clearly, $G$ is connected.  Let us now argue that it does not contain a vertex cut of size one.
Suppose for a contradiction $G=G_1\cup G_2$, where $G_1$ and $G_2$ are proper induced subgraphs of $G$
intersecting in a single vertex $v$.  Let us first consider the case that $H\subseteq G_1$.
For $i\in\{1,2\}$, we define a scene $S_i=(G_i,H_i,L_1,L_2,\varphi_0\restriction V(G_i),\vec{\delta}^i)$
as follows.  Let $H_1=H$ and $\vec{\delta}^1=\vec{\delta}$.  By the minimality of $S$,
$S_1$ has a trajectory $\vec{\psi}$.  Let $H_2=v$ and $\vec{\delta}^2=\vec{\psi}\restriction H_2$.
Note that $S_2$ is valid, and by the minimality of $S$, $S_2$ has a trajectory.
Lemma~\ref{lem:total} thus implies that $S$ has a trajectory, which is a contradiction.

Therefore, we can assume that $H\not\subseteq G_1$, and symmetrically, $H\not\subseteq G_2$.
Consequently, $H$ is a $3$-vertex path, $v$ is the middle vertex of $H$, and for $i\in\{1,2\}$,
$H\cap G_i$ is a $2$-vertex path.  For $i\in\{1,2\}$, let $S_i$ be the scene $(G_i,H\cap G_i, L_1,L_2, \varphi_0\restriction V(G_i),\vec{\delta}\restriction (H\cap G_i))$.
Note that $S_i$ is valid, and by the minimality of $S$, it has a trajectory $\vec{\psi}^i$ witnessed by $(\sigma_{i,1},\sigma_{i,2})$.
Let $\vec{\pi}$ be a witness of the trajectory $\vec{\delta}$.
For $j\in\{1,2\}$, let $\varphi_j=\psi^1_j\cup \psi^2_j$ and $\sigma_j=\sigma_{1,j}^{G_1-V(H)}+\sigma_{2,j}^{G_2-V(H)}+\pi_j$.
Then $(\sigma_1,\sigma_2)$ is a witness that $\vec{\varphi}$ is a trajectory of $S$, which is a contradiction.

The argument from the first paragraph can also be applied in the following cases:
\begin{itemize}
\item The outer face of $G$ has a chord $uv$ and, writing $G=G_1\cup G_2$ for proper induced subgraphs
$G_1$ and $G_2$ of $G$ intersecting in $uv$, we have $H\subseteq G_1$.
\item A triangle $T$ in $G$ does not bound a face, $G_2$ is the subgraph of $G$ drawn in the closed disk bounded
by $T$, and $G_1=G-(V(G_2)\setminus V(T))$.  To see that $S_2$ is valid, note that the condition (Gc) holds
since all vertices $x\in V(G_2)\setminus V(T)$ satisfy $|L_1(x)|\ge 10$.
\end{itemize}
Therefore, we can assume that every triangle in $G$ bounds a face, and
if the outer face of $G$ has a chord, then $H$ is a $3$-vertex path and every chord of the outer
face of $G$ is incident with the middle vertex of $H$.

Finally, consider the case $H$ is not a $3$-vertex path.  Then either $H$ is a path with at most two
vertices, or a triangle bounding the outer face of $G$.  In the latter case, we obtain contradiction
by considering the scene $S'=(G-e,H-e,L_1,L_2,\varphi_0,\vec{\delta})$ for an edge $e$ of $H$;
note that by Observation~\ref{obs:reorder}, the order of recolorings on $H-e$ can be chosen to match the one
in a witness of $\vec{\delta}$ on $H$, ensuring that a trajectory of $S'$ is also a trajectory of $S$.
In the former case, let $v$ be a vertex in the outer face of $G$ not belonging to $H$ and (unless $V(H)=\emptyset$)
adjacent to a vertex of $H$.  Let $H'=G[V(H)\cup\{v\}]$.  Note that since the outer face of $G$ does not
have a chord, if $|V(H')|=3$, then any vertex $x$ adjacent to all vertices of $H'$ is internal and satisfies $|L_1(x)|\ge 10$.
Let us now extend $\vec{\delta}$ to an $(L_1,L_2)$-trajectory $\vec{\psi}$ of $H'$ by setting $\psi_0(v)=\varphi_0(v)$
and for $i\in\{1,2\}$, letting $\psi_i(v)$ be a color in $L_i(v)\setminus\bigcup_{u\in V(H)} \{\delta_{i-1}(u),\delta_i(u)\}$;
such a color exists, since $|L_i(v)|>4$.  This is indeed a trajectory, since we can recolor $v$ to $\psi_i(v)$ before
performing the rest of the recolorings from $\psi_{i-1}$ to $\psi_i$ (and thus both sequences of the resulting witness
are $H$-late).  A trajectory for $(G,H',L_1,L_2,\varphi_0,\vec{\psi})$ is also a trajectory for $S$, contradicting
the minimality of $S$.
\end{proof}

Actually, we can eliminate all chords of the outer face of a minimal counterexample.
\begin{lemma}\label{lemma:indu}
Suppose $S=(G,H,L_1,L_2,\varphi_0,\vec{\delta})$ is a minimal counterexample.
Then the outer face of $G$ is bounded by an induced cycle.
\end{lemma}
\begin{proof}
Suppose for contradiction this is not the case.  By Lemma~\ref{lemma:basprop}, this implies
$H=h_1h_2h_3$ is a $3$-vertex path and the outer face of $G$ has a chord $h_2v$.
Let $G=G_1\cup G_2$, where $G_1$ and $G_2$ are proper induced subgraphs of $G$
intersecting in $h_2v$ and with $h_1\in V(G_1)$.  Let $K$ be the cycle bounding the outer face of $G$.

Let $\vec{\psi}$ be an $(L_1,L_2)$-trajectory in $H'=G[V(H)\cup\{v\}]$ extending $\vec{\delta}$ chosen
as follows.  We set $\psi_0(v)=\varphi_0(v)$.  For $i=1,2$, let $\psi_i(v)\in L_i(v)$ be chosen
so that
\begin{itemize}
\item[(i)] if $uv\in E(G)$ for $u\in V(H)$, then $\psi_i(v)\not\in\{\delta_{i-1}(u),\delta_i(u)\}$, and
\item[(ii)] if $i=1$, $z$ is a common neighbor of $v$, $h_2$, and $h_j$ for some $j\in\{1,3\}$,
and $L_1(z)=\{\varphi_0(v),\delta_0(h_2),\delta_1(h_2),\delta_0(h_j),\delta_1(h_j), c\}$,
then $\psi_1(v)\neq c$.
\end{itemize}
By planarity and since $|L_1(z)|=6$ is only possible for vertices $z\in V(K)$, note that in the
latter case, $vh_j\not\in E(G)$.  Therefore, these conditions either forbid at most five colors
in total, or $v$ is a common neighbor of all vertices of $H$ and the condition (i) forbids
six colors in total.  It follows that $\psi_i(v)$ can be chosen as described, since $|L_1(v)|\ge 6$,
$|L_2(v)|\ge 7$, and not all colors in $L_1(v)$ can be forbidden due to the assumption (Gc).
Note that the condition (i) ensures that $\vec{\delta}$ lifts to the trajectory $\vec{\psi}$ in $H'$
(for $i\in\{1,2\}$, we can recolor $v$ before the recolorings from $\delta_{i-1}$ to $\delta_i$);
let $(\pi_1,\pi_2)$ witness this.

For $j\in\{1,2\}$, let $S_j=(G_j,H'\cap G_j, L_1, L_2, \varphi_0\restriction V(G_j), \vec{\psi}\restriction (H'\cap G_j))$.
Note $S_j$ is a valid scene; the condition (ii) ensures $S_j$ satisfies (Gc).  By the minimality
of $S$, $S_j$ has a trajectory $\vec{\psi}^j$, witnessed by $(\sigma_{j,1},\sigma_{j,2})$.
For $i\in\{1,2\}$, let $\sigma_i=\sigma^{G_1-V(H')}_{1,i}+\sigma^{G_2-V(H')}_{2,i}+\pi_i$.
Then $(\sigma_1,\sigma_2)$ is a witness of a trajectory of $S$, which is a contradiction.
\end{proof}

Now, we restrict adjacencies of the internal vertices to vertices incident with the outer face.
\begin{lemma}\label{lemma:few}
Suppose $S=(G,H,L_1,L_2,\varphi_0,\vec{\delta})$ is a minimal counterexample, $H=h_1h_2h_3$, $K$ is the cycle
bounding the outer face of $G$, and $v$ is an internal vertex of $G$.  Then $v$ does not have two non-adjacent
neighbors in $K-h_2$.
\end{lemma}
\begin{proof}
Suppose for contradiction that $G=G_1\cup G_2$, where $G_1$ and $G_2$ are proper induced subgraphs of $G$
intersecting in an induced path $v_1vv_2$, with $v_1,v_2\in V(K-h_2)$ and $H\subset G_1$.

Let $L'_1$ be the list assignment obtained from $L_1$ by changing the list of $v$
as follows.  If $v$, $v_1$, and $v_2$ have a common neighbor $z\in V(K-h_2)$, then let $C$
be a set of three colors belonging to $L_1(z)\setminus\{\varphi_0(v),\varphi_0(v_1),\varphi_0(v_2)\}$,
and let $L'_1(v)=L_1(v)\setminus C$, otherwise let $L'_1(v)=L_1(v)$.
Note that $|L'_1(v)|\ge 7$ and $|L'_1(v)\setminus L_2(v)|\le |L_1(v)\setminus L_2(v)|\le 1$.
Hence, $S_1=(G_1,H,L'_1,L_2,\varphi_0,\vec{\delta})$ is a valid scene, and by the minimality of $S$,
it has a trajectory $\vec{\psi}$.

Consider the scene $S_2=(G_2,v_1vv_2,L_1,L_2,\vec{\psi}\restriction v_1vv_2)$.
We claim this scene is valid.  It suffices to verify the condition (Gc) in the case $v$, $v_1$, and $v_2$ have a common neighbor $z\in V(K-h_2)$.
To do so, note that
$$L_1(z)\setminus \{\varphi_0(v),\psi_1(v),\varphi_0(v_1),\psi_1(v_1),\varphi_0(v_2),\psi_1(v_2)\}\supseteq C\setminus\{\psi_1(v),\psi_1(v_1),\psi_1(v_2)\}\neq\emptyset,$$
since $|C|=3$, $\psi_1(v)\in L'_1(v)$, and $L'_1(v)\cap C=\emptyset$.  Therefore, by the minimality of $S$, $S_2$ has a trajectory.
Lemma~\ref{lem:total} implies $S$ has a trajectory, which is a contradiction.
\end{proof}

Let us deal with one more simple case.
\begin{lemma}\label{lemma:long}
If $S=(G,H,L_1,L_2,\varphi_0,\vec{\delta})$ is a minimal counterexample, then the cycle $K$ bounding the outer
face of $G$ has length at least five.
\end{lemma}
\begin{proof}
By Lemma~\ref{lemma:basprop}, $H=h_1h_2h_3$ is a $3$-vertex induced path, and thus $|K|\ge 4$.
Suppose for a contradiction $K=h_1h_2h_3v$.  For $i\in \{1,2\}$, choose
$c_i\in L_i(v)\setminus \{\delta_{i-1}(h_1),\delta_i(h_1), \delta_{i-1}(h_3),\delta_i(h_3)\}$ arbitrarily.
Let $c_0=\varphi_0(v)$.
Let $L'_1$ and $L'_2$ be list assignments obtained from $L_1$ and $L_2$ by, for each internal vertex $u$ adjacent to $v$,
setting $L'_1(u)=L_1(u)\setminus\{c_0,c_1,c_2\}$ and $L'_2(u)=L_2(u)\setminus\{c_1,c_2\}$.
Note that $|L'_1(u)|,|L'_2(u)|\ge 7$ and $|L'_1(u)\setminus L'_2(u)|\le |L_1(u)\setminus L_2(u)|\le 1$,
and thus $S'=(G-v, H, L'_1, L'_2, \varphi_0,\vec{\delta})$ is a valid scene.  By the minimality of $S$,
$S'$ has a trajectory $\vec{\psi}$ witnessed by sequences $(\pi_1,\pi_2)$.
For $i\in\{1,2\}$, let $\sigma_i=\pi_i^{G-(V(H)\cup\{v\})}+(v,c_i)+\pi_i^H$.
Note that the sequence $\pi_i^{G-(V(H)\cup\{v\})}$ does not recolor any neighbor $u$ of $v$
to $c_{i-1}$ (the current color of $v$) or to $c_i$ (the color $v$ is to be recolored to), since
$c_{i-1},c_i\not\in L'_i(u)$.  Moreover, after recoloring $v$ to $c_i$, we can recolor $H$ from $\delta_{i-1}$ to $\delta_i$
by the sequence $\pi_i^H$ by the choice of $c_i$.  Therefore, $(\sigma_1,\sigma_2)$ witnesses a trajectory of $S$,
which is a contradiction.
\end{proof}

Consider a graph $K$ and its induced subgraph $H$.  We say that a trajectory $\vec{\psi}$ in $K$ is \emph{stable outside of $H$}
if $\psi_2(v)=\psi_1(v)$ for every $v\in V(K)\setminus V(H)$, i.e., the coloring outside of $H$ only changes once,
from $\psi_0\restriction (V(K)\setminus V(H))$ to $\psi_1\restriction (V(K)\setminus V(H))$.
As a final preparatory step, let us show the existence of a stable trajectory for the outer face of a minimal counterexample.
\begin{lemma}\label{lemma:stable}
Let $S=(G,H,L_1,L_2,\varphi_0,\vec{\delta})$ be a minimal counterexample and let $K$ be the cycle bounding the outer
face of $G$.  Then $\vec{\delta}$ lifts to an $(L_1,L_2)$-trajectory in $K$ starting with $\varphi_0\restriction V(K)$ and stable outside of $H$.
\end{lemma}
\begin{proof}
For each $v\in V(K)\setminus V(H)$, let $L(v)=L_1(v)\cap L_2(v)$.  Since $|L_1(v)\setminus L_2(v)|\le 1$,
we have $|L(v)|=|L_1(v)|-|L_1(v)\setminus L_2(v)|\ge 5$.  Let $K=h_1h_2h_3v_1v_2\ldots v_m$,
where $m\ge 2$ by Lemma~\ref{lemma:long}.  Let us define a coloring $\psi_1$ of $K$ extending $\delta_1$
as follows.  Choose $\psi_1(v_1)\in L(v_1)\setminus \{\delta_0(h_3),\delta_1(h_3),\delta_2(h_3),\varphi_0(v_2)\}$
arbitrarily.  For $i=2,\ldots, m-1$, choose $\psi_1(v_i)\in L(v_i)\setminus \{\psi_1(v_{i-1}), \varphi_0(v_{i+1})\}$ arbitrarily.
Finally, choose $\psi(v_m)\in L(v_m)\setminus\{\psi_1(v_{m-1}),\delta_0(h_1),\delta_1(h_1),\delta_2(h_1)\}$ arbitrarily.
We can recolor $K$ from $\varphi_0\restriction V(K)$ to $\psi_1$ by recoloring $v_1$, \ldots, $v_m$
in order, followed by recoloring of $H$ from $\delta_0$ to $\delta_1$.  Moreover, the choice of $\psi_1(v_1)$ and $\psi_1(v_m)$
ensures this can be followed by recoloring of $H$ from $\delta_1$ to $\delta_2$.  This gives a witness of
an $(L_1,L_2)$-trajectory in $K$ starting with $\varphi_0\restriction V(K)$ and stable outside of $H$.
\end{proof}

We are now ready to prove the main list-recoloring result.
\begin{theorem}\label{thm:9colors}
Every valid scene has a trajectory.
\end{theorem}
\begin{proof}
Suppose for a contradiction this is not the case.  Then there exists a counterexample, and thus
also a minimal one, say $S=(G,H,L_1,L_2,\varphi_0,\vec{\delta})$.
Let $K$ be the induced cycle bounding the outer face of $G$.
By Lemma~\ref{lemma:stable}, $\vec{\delta}$ lifts to an $(L_1,L_2)$-trajectory $\vec{\psi}$ in $K$ starting with
$\varphi_0\restriction V(K)$ and stable outside of $H$; let $(\pi_1,\pi_2)$ be the corresponding witness,
where $\pi_2=\pi_2^H$.

Let $G'=G-(V(K)\setminus V(H))$.  Let $L'_1$ and $L'_2$ be list assignments for $G'$ matching $L_1$ and $L_2$ on $H$ and such that
for each vertex $u\in V(G)\setminus V(K)$,
$$L'_1(u)=L_1(u)\setminus \bigcup_{v\in N(u)\cap V(K)\setminus V(H)} \{\varphi_0(v),\psi_1(v)\}$$
and
$$L'_2(u)=L_2(u)\setminus \bigcup_{v\in N(u)\cap V(K)\setminus V(H)} \{\psi_1(v)\}.$$
By Lemma~\ref{lemma:few}, $u$ has at most two consecutive neighbors in $K-V(H)$, and thus
$|L'_1(u)|\ge 6$ and $|L'_2(u)|\ge 7$.  Moreover, $u$ is not adjacent to all vertices of $H$ by the same lemma,
implying (Gc) holds.  Therefore, $S'=(G',H,L'_1,L'_2,\varphi_0\restriction V(G'),\vec{\delta})$ is a valid scene.
By the minimality of $G$, $S'$ has a trajectory $\vec{\theta}$, witnessed by $(\omega_1,\omega_2)$.

For $i\in \{1,2\}$, let $\sigma_i=\omega_i^{G'-V(H)}+\pi_i$.  Note that the choice of the list assignment $L'_1$ implies
the recoloring according to $\omega_1^{G'-V(H)}$ does not give any vertex $u\in V(G)\setminus V(K)$ the color
$\varphi_0(v)$ of any neighbor $v\in V(K)$, and thus the sequence of recolorings $\omega_1^{G'-V(H)}$ is proper.
Moreover, for any such $u$ and $v$ the choice of $L'_1$ implies $\theta_1(u)\neq \psi_1(v)$, and thus the sequence of recolorings $\pi_1$ 
is proper.  The choice of $L'_2$ similarly shows that the recoloring according to the sequence $\omega_2^{G'-V(H)}$ does not conflict
with $\psi_1$.  The sequence of recolorings $\pi_2$ only affects the vertices of $H$ and does not conflict with $\theta_2$, since $\omega_2$
performs a once-only sequence of recolorings with the same result.
Therefore, $(\sigma_1,\sigma_2)$ is a witness of a trajectory of $S$, which is a contradiction.
\end{proof}

\section{Triangle-free graphs}

Next, let us give the argument for triangle-free graphs.
We say that a scene $S=(G,H,L_1,L_2,\varphi_0,\vec{\delta})$ is \emph{T-valid} if $G$
is triangle-free, $|V(H)|\le 5$,
\begin{itemize}
\item[(Ta)] for each $v\in V(G)$, $|L_1(v)\setminus L_2(v)|\le 1$ and $|L_1(v)|,|L_2(v)|\ge 5$,
\item[(Tb)] if $v\in V(G)$ is an internal vertex, then $|L_1(v)|\ge 7$ and $|L_2(v)|\ge 6$.
Moreover, if $H$ is a path with at least four vertices, then at least $|V(H)|-3$ ends $x$ of $H$
have the following property: All neighbors $v\in V(G)\setminus V(H)$ of $x$ satisfy $|L_1(v)|\ge 7$ and $|L_2(v)|\ge 6$.  And,
\item[(Tc)] if $H$ is a path with five vertices, then each vertex in $V(G)\setminus V(H)$ has
at most one neighbor in $H$.
\end{itemize}
Again, we aim to prove by contradiction that every T-valid scene has a trajectory.
We overload the terminology; in this section, a \emph{counterexample} is a T-valid scene without a trajectory.
The argument from the proof of Lemma~\ref{lemma:basprop} gives the following.
\begin{lemma}\label{lemma:basprop-g5}
Suppose $S=(G,H,L_1,L_2,\varphi_0,\vec{\delta})$ is a minimal counterexample.
Then $G$ is $2$-connected, $H$ is a path with at least three vertices, every $(\le\!5)$-cycle
in $G$ bounds a face, and every chord of the outer face of $G$ is incident with one of the non-end vertices of $H$.
Moreover, every vertex in $V(G)\setminus V(H)$ has degree at least three.
\end{lemma}
\begin{proof}
The claims on connectivity and chords are proved exactly as in Lemma~\ref{lemma:basprop}.
Non-facial $(\le\!5)$-cycles are excluded in the same way as non-facial triangles in Lemma~\ref{lemma:basprop}.

Suppose now that $v\in V(G)\setminus V(H)$ has degree at most two.  By the minimality of $S$,
the scene $(G-v,H,L_1,L_2,\varphi_0\restriction V(G-v),\vec{\delta})$ has a trajectory $\vec{\psi}$ witnessed by $(\sigma_1,\sigma_2)$.
For $i\in\{1,2\}$, choose a color $c_i\in L_i(v)\setminus \bigcup_{uv\in E(G)} \{\psi_{i-1}(u),\psi_i(u)\}$,
and let $\sigma'_i=(v,c_i)+\sigma_i$.  Then $(\sigma'_1,\sigma'_2)$ witnesses a trajectory of $S$, which is a contradiction.
Therefore, every vertex in $V(G)\setminus V(H)$ has degree at least three.

If $|V(H)|\le 2$, we can extend $H$ to a $3$-vertex path by adding vertices and straightforwardly lifting $\vec{\delta}$.
Suppose now that $H$ is a cycle.  As in Lemma~\ref{lemma:basprop}, this is dealt with by deleting an edge of the cycle,
but a little care is needed to ensure (Tc) holds.  Specifically, we note that if a vertex $v$ had two neighbors in $H$,
then since all $(\le\!5)$-cycles bound faces, we would have $V(G)=V(H)\cup\{v\}$ and $\deg v=2$, contradicting the previous
paragraph.
\end{proof}

Similarly to Lemma~\ref{lemma:indu}, we can eliminate the remaining chords.
\begin{lemma}\label{lemma:indu-g5}
Suppose $S=(G,H,L_1,L_2,\varphi_0,\vec{\delta})$ is a minimal counterexample.
Then the outer face of $G$ is bounded by an induced cycle.
\end{lemma}
\begin{proof}
Suppose for contradiction this is not the case.  By Lemma~\ref{lemma:basprop-g5}, this implies
we can label the vertices of the path $H$ as $H=h_1\ldots h_m$ (for some $m\in\{3,4,5\}$) 
and the outer face of $G$ has a chord $h_tv$ for some $t\in\{2,\ldots,\lceil m/2\rceil\}$.
Let $G=G_1\cup G_2$, where $G_1$ and $G_2$ are proper induced subgraphs of $G$
intersecting in $h_2v$ and with $h_1\in V(G_1)$.  If $t=3$ (and so $m=5$), then $v$ has only one neighbor in $H$
by (Tc), and thus $vh_1\not\in E(G)$.  If $t=2$, this is also true, since $G$ is triangle-free.
Let $(\beta_1,\beta_2)$ be a witness of the trajectory $\vec{\delta}$.

By the minimality of $S$, the scene $S_2=(G_2,H\cap G_2,L_1,L_2,\varphi_0\restriction V(G_2),\vec{\delta}\restriction (H\cap G_2))$
has a trajectory $\vec{\psi}$, witnessed by $(\omega_1,\omega_2)$.
Let $Q=h_1\ldots h_tv$ and let $\vec{\gamma}$ be the trajectory matching $\vec{\delta}$ on $Q-v$ and
$\vec{\psi}$ on $v$.  Note that $\vec{\gamma}$ is indeed a trajectory on $Q$, as the vertex $v$ can be recolored
just before $h_t$, matching the order in the $(H\cap G_2)$-late recoloring sequences of $(\omega_1,\omega_2)$.
The scene $S_1=(G_1,Q,L_1,L_2,\varphi_0\restriction V(G_1),\vec{\gamma})$
is T-valid, and by the minimality of $S$, it has a trajectory $\vec{\theta}$ witnessed by $(\pi_1,\pi_2)$.
For $i\in\{1,2\}$, let $\sigma_i=\pi_i^{G_1-V(Q)}+\omega_i^{G_2-V(H)}+\beta_i$.
Then $(\sigma_1,\sigma_2)$ is a witness of a trajectory of $S$, which is a contradiction.
\end{proof}

The proof of the following claim is analogous to the proof of Lemma~\ref{lemma:few} (but simpler, since we do not
need to worry about the condition (Gc)).  Note also that unlike Lemma~\ref{lemma:few}, $v$ cannot have two adjacent
neighbors, since $G$ is triangle-free.
\begin{lemma}\label{lemma:few-g5}
Suppose $S=(G,H,L_1,L_2,\varphi_0,\vec{\delta})$ is a minimal counterexample, $H=h_1\ldots h_m$, $K$ is the cycle
bounding the outer face of $G$, and $v$ is an internal vertex of $G$.  Then $v$ has at most one 
neighbor in $K-\{h_2,\ldots, h_{m-1}\}$.
\end{lemma}

Next, we exclude the possibility that $G$ has a short outer face.
\begin{lemma}\label{lemma:long-g5}
If $S=(G,H,L_1,L_2,\varphi_0,\vec{\delta})$ is a minimal counterexample, then the cycle $K$ bounding the outer
face of $G$ has length at least six.
\end{lemma}
\begin{proof}
Suppose for a contradiction $|V(K)|\le 5$.  It is easy to see that $\vec{\delta}$ lifts to an $(L_1, L_2)$-trajectory $\vec{\psi}$ in $K$
starting with $\varphi_0\restriction V(K)$.   By the minimality of $S$, the scene
$(G,K,L_1,L_2,\varphi_0,\vec{\psi})$ has a trajectory. This trajectory is also a trajectory of $S$, a contradiction.
\end{proof}

As a final preparatory step, let us show the existence of a stable trajectory; this is substantially more involved than in the proof of Lemma~\ref{lemma:stable}.
\begin{lemma}\label{lemma:stable-g5}
Let $S=(G,H,L_1,L_2,\varphi_0,\vec{\delta})$ be a minimal counterexample and let $K$ be the cycle bounding the outer
face of $G$.  Then $\vec{\delta}$ lifts to an $(L_1,L_2)$-trajectory in $K$ starting with $\varphi_0\restriction V(K)$ and stable outside of $H$.
\end{lemma}
\begin{proof}
For each $v\in V(K)\setminus V(H)$, let $L(v)=L_1(v)\cap L_2(v)$.  Since $|L_1(v)\setminus L_2(v)|\le 1$,
we have $|L(v)|\ge |L_1(v)|-1$.
Let $H=h_1\ldots h_m$ and $K=h_1\ldots h_mv_1\ldots v_t$.
Let $(\beta_1,\beta_2)$ be the witness of the trajectory $\vec{\delta}$.
By Lemma~\ref{lemma:long-g5} and (Tc), we have $t\ge 2$.
Let $L'(v_1)=L(v_1)\setminus \{\delta_0(h_m),\delta_1(h_m),\delta_2(h_m)\}$ and
$L'(v_t)=L(v_t)\setminus \{\delta_0(h_1),\delta_1(h_1),\delta_2(h_1)\}$.

If $t=2$, then we have $m\ge 4$ by Lemma~\ref{lemma:long-g5}.
By (Tb) and symmetry, we can assume $|L(v_1)|\ge 6$, and thus $|L'(v_1)|\ge 3$.
Note that $|L(v_2)|\ge 4$ and $|L'(v_2)|\ge 1$.  Choose $c_2\in L'(v_2)$
and $c_1\in L'(v_1)\setminus \{\varphi_0(v_2),c_2\}$.
The desired trajectory is obtained by recoloring according to the sequence $(v_1,c_1),(v_2,c_2)+\beta_1$.
Hence, we can assume $t\ge 3$.  Note we have $|L'(v_1)|,|L'(v_t)|\ge 1$ and $|L(v_j)|\ge 4$ for $j\in\{2,\ldots,t-1\}$.
Choose $c_1\in L'(v_1)$ and $c_t\in L'(v_t)$ arbitrarily.

Suppose now $t=3$. If $c_1\neq\varphi_0(v_2)$, choose $c_2\in L(v_2)\setminus\{c_1,c_3,\varphi_0(v_3)\}$
and recolor according to the sequence $(v_1,c_1),(v_2,c_2),(v_3,c_3)+\beta_1$.
Hence, we can assume $c_1=\varphi_0(v_2)$, and by symmetry, $c_3=\varphi_0(v_2)$.
Choose $c_2\in L(v_2)\setminus\{c_1,\varphi_0(v_1),\varphi_0(v_3)\}$ and
recolor according to the sequence $(v_2,c_2),(v_1,c_1),(v_3,c_3)+\beta_1$.

Finally, suppose $t\ge 4$.  Choose $c_2\in L(v_2)\setminus\{c_1,\varphi_0(v_1),\varphi_0(v_3)\}$,
for $j=3,\ldots,t-2$, choose $c_j\in L(v_j)\setminus\{c_{j-1},\varphi_0(v_{j+1})\}$, and
choose $c_{t-1}\in L(v_{t-1})\setminus\{c_{t-2},c_t,\varphi_0(v_t)\}$ arbitrarily.
We perform the recoloring according to the sequence $(v_2,c_2),(v_3,c_3),\ldots,(v_t,c_t),(v_1,c_1)+\beta_1$.
\end{proof}

We are now ready to prove our main list-recoloring result for triangle-free graphs.
\begin{theorem}\label{thm:6colors}
Every T-valid scene has a trajectory.
\end{theorem}
\begin{proof}
Suppose for a contradiction this is not the case, and thus there exists a minimal counterexample $S=(G,H,L_1,L_2,\varphi_0,\vec{\delta})$.
Let $K$ be the induced cycle bounding the outer face of $G$.
By Lemma~\ref{lemma:stable-g5}, $\vec{\delta}$ lifts to an $(L_1,L_2)$-trajectory $\vec{\psi}$ in $K$ starting with
$\varphi_0\restriction V(K)$ and stable outside of $H$; let $(\pi_1,\pi_2)$ be a corresponding witness,
where $\pi_2=\pi_2^H$.

Let $G'=G-(V(K)\setminus V(H))$.  Let $L'_1$ and $L'_2$ be list assignments on $G'$ matching $L_1$ on $K$ and such that
for each vertex $u\in V(G)\setminus V(K)$,
$$L'_1(u)=L_1(u)\setminus \bigcup_{v\in N(u)\cap V(K)\setminus V(H)} \{\varphi_0(v),\psi_1(v)\}$$
and
$$L'_2(u)=L_2(u)\setminus \bigcup_{v\in N(u)\cap V(K)\setminus V(H)} \{\psi_1(v)\}.$$
By Lemma~\ref{lemma:few-g5}, the vertex $u$ has at most one neighbor in $K-V(H)$, and thus
$|L'_1(u)|,|L'_2(u)|\ge 5$.  Moreover, if $u$ has a neighbor in $V(K)\setminus V(H)$,
then it is not adjacent to the ends of $H$ by the same lemma,
and thus the scene $S'=(G',H,L'_1,L'_2,\varphi_0\restriction V(G'),\vec{\delta})$ satisfies (Ta) and (Tb).
Finally, $S'$ satisfies (Tc) because (Tc) holds in $S$.  Therefore, $S'$ is a T-valid scene.
By the minimality of $G$, $S'$ has a trajectory $\vec{\theta}$, witnessed by $(\omega_1,\omega_2)$.

For $i\in \{1,2\}$, let $\sigma_i=\omega_i^{G'-V(H)}+\pi_i$.  Note that the choice of the list assignment $L'_1$ implies
the recoloring according to $\omega_1^{G'-V(H)}$ does not give any vertex $u\in V(G)\setminus V(K)$ the color
$\varphi_0(v)$ of any neighbor $v\in V(K)$, and that $\theta_1(u)\neq \psi_1(v)$, so the recoloring according to $\pi_1$ can be
performed.  The choice of $L'_2$ similarly shows that the recoloring according to $\omega_2^{G'-V(H)}$ does not conflict
with $\psi_1$.  The recoloring $\pi_2$ only affects vertices of $H$ and does not conflict with $\theta_2$, since $\omega_2$
performs a once-only sequence of recolorings with the same result.
Therefore, $(\sigma_1,\sigma_2)$ is a witness of a trajectory of $S$, which is a contradiction.
\end{proof}

\section{Eliminating a color}

We now apply Theorems~\ref{thm:9colors} and \ref{thm:6colors} to prove the two cases of
Theorem~\ref{thm:col}.

\begin{proof}[Proof of Theorem~\ref{thm:col}]
Let $L_2(v)=L(v)\setminus f(v)$ for each $v\in V(G)$.
Then $S=(G,\varnothing,L,L_2,\varphi,\varnothing)$ is T-valid if $G$ is triangle-free and valid otherwise.
By Theorems~\ref{thm:9colors} and \ref{thm:6colors}, $S$ has a trajectory $\vec{\psi}$.
A witness of this trajectory gives a sequence of recolorings from $\varphi$ to $\varphi'=\psi_2$,
where $\varphi'(v)\neq f(v)$ for each $v\in V(G)$ by the choice of $L_2$, and each vertex is recolored at most twice.
\end{proof}

Note that while Theorem~\ref{thm:col} works in a general list coloring setting,
we can only apply it to reconfiguration of ordinary proper colorings.  Indeed, in the list coloring setting,
it may not be the case that the same color is available at each vertex of the chosen independent set,
and thus the proof of Lemma~\ref{lemma:freeup} fails.  Nevertheless, we find it likely that Theorems~\ref{thm:main}
and~\ref{thm:trmain} generalize to list coloring.
\begin{conjecture}
Let $G$ be a planar graph, let $L$ be a list assignment for $G$,
and let $\varphi$ and $\varphi'$ be $L$-colorings of $G$.
If either
\begin{itemize}
\item[(a)] $|L(v)|\ge 10$ for every $v\in V(G)$, or
\item[(b)] $G$ is triangle-free and $|L(v)|\ge 7$ for every $v\in V(G)$,
\end{itemize}
then $\varphi$ can be transformed to $\varphi'$ by $O(n)$ recolorings so that
all intermediate colorings are proper $L$-colorings.
\end{conjecture}

As mentioned in~\cite{DF}, it is possible that the number $10$ of colors in the statement of Theorem \ref{thm:main} is not the best possible. Perhaps a more sophisticated multiphase recoloring process might allow one to replace $10$ by a smaller integer and still obtain a linear bound on the diameter of the reconfiguration graph. To be slightly more precise, in such a setting we have a planar graph $G$ and list assignments $L_1, \dots, L_m$ for $G$ (with various well-chosen conditions on the lists). As long as $m$ is fixed, the aim would be to show the existence of an $(L_1, \dots, L_m)$-trajectory in $G$ starting with a $k$-coloring $\varphi_0$ for some $k \in \{7, 8, 9\}$ (that is, a tuple $\vec{\varphi}=(\varphi_0,\varphi_1, \dots, \varphi_m)$
where for $i\in\{1,\dots, m\}$, $\varphi_i$ is an $L_i$-coloring of $G$
obtained from $\varphi_{i-1}$ by a once-only sequence $\sigma_i$ of recolorings), such that $\varphi_m$ uses one less color than $\varphi_0$. Notice that if $\varphi_0$ is a $9$-coloring, then such a result would straightforwardly give a linear bound on the diameter of $R_9(G)$ (by essentially the proof of Theorem \ref{thm:main}). However, as soon as $9$ is replaced by $8$, one can no longer apply Theorem \ref{thm:bousquet} and hence more ideas would be needed.

\section*{Acknowledgments}

Zden\v{e}k Dvo\v{r}\'ak was supported in part by ERC Synergy grant DYNASNET no. 810115.
Carl Feghali was supported by grant 19-21082S of the Czech Science Foundation.

\bibliography{bibliography}{}
\bibliographystyle{abbrv}
 
\end{document}